\documentclass[12pt]{amsart}
\usepackage{amssymb}

 \newtheorem{thm}{Theorem}[section]
 \newtheorem{cor}[thm]{Corollary}
 \newtheorem{lem}[thm]{Lemma}
 
 \newtheorem{claim}[thm]{Claim}
 \newtheorem{defn}[thm]{Definition}
 \newtheorem{rem}[thm]{Remark}
 \numberwithin{equation}{section}
 \newcommand \be     {\begin{equation}}
\newcommand \ee     {\end{equation}}

\newcommand \Mcal   {\mathcal M}
 \newcommand {\RR}{\mathbb{R}}
 \newcommand {\RP}{\mathbb{R}_+}
 \newcommand \Rn    {\mathbb{R}^n}
 \newcommand {\Phim} {\Phi_\eta}
 
 \newcommand {\ueps}{u^\varepsilon}

 \newcommand {\uupk}{u^{(k)}}
 \newcommand {\uzupk}{u_0^{(k)}}

  \newcommand{\eps}{\varepsilon}
 \newcommand{\Geps}{G_\eps}
 \newcommand {\Thetam}{\Theta_\eta}
 \newcommand{\set}[1]{\left\{#1\right\}}

\begin{document}

\title[ROUGH SOLUTIONS VISCOUS CONSERVATION LAWS]
 { VISCOUS CONSERVATION LAWS in 1D with MEASURE INITIAL DATA}

\author{ Miriam Bank }
\address{ Miriam Bank:Azrieli College of Engineering, Jerusalem 91035, Israel}
\author{ Matania Ben-Artzi}
\address{ Matania Ben-Artzi:Institute of Mathematics, Hebrew University, Jerusalem
91904, Israel}
 \author{ Maria E. Schonbek }
 \address{ Maria E. Schonbek:Department of Mathematics, UC Santa Cruz, Santa Cruz,
CA 95064, USA}



\thanks{We are grateful to Prof. M. Slemrod for many useful discussions.}


\subjclass[2010]{Primary 35K15; Secondary 35K59}

\keywords{scalar conservation law, viscosity, measure initial data, p-condition, sup-norm estimates,decay estimates}

\date{\today}




\begin{abstract} The one-dimensional viscous conservation law is considered on the whole line
$$
      u_t + f(u)_x=\eps u_{xx},\quad (x,t)\in\RR\times\overline{\RP},\quad
   \eps>0,
   $$
 subject to positive  measure initial data.

 The flux $f\in C^1(\RR)$ is assumed to satisfy a $p-$condition, a weak form of convexity.

  Existence and uniqueness of solutions is established. The method of proof relies on sharp decay estimates for viscous Hamilton-Jacobi equations.

\end{abstract}

\maketitle

   \vspace{.5in}
   \section{INTRODUCTION}

   We consider here the (viscous) nonlinear scalar conservation law in one space dimension, for an unknown real function $u(x,t),$
   \begin{equation}\label{generaleq}
      u_t + f(u)_x=\eps u_{xx},\quad (x,t)\in\RR\times\overline{\RP},\quad
   \eps>0,
   \end{equation}
   subject to the initial condition
     \begin{equation}
     \label{initialgen}
      u(x,0)=u_0(x)\in\Mcal_+,
     \end{equation}
      where $\Mcal_+=\Mcal_+(\RR)$ is the space of (finite) nonnegative Borel measures on the line.

  \be\label{eqassumefr}\aligned     \mbox{We  assume that}\,\, f\in C^1(\RR),\,f(0)=f'(0)=0,\\
      f'(u)\,\,\mbox{is locally H\"{o}lder for }\,\,u\in\RR.
  \endaligned\ee

        Throughout the paper we fix $\eps>0,$ and omit the obvious dependence of the solution on this parameter
     (namely, we write $u$ and not $\ueps$).

    In particular, we are interested in the case
     \be\label{eqmdeltainit} u_0(x)=M\delta_{x_0},\ee
      where $M>0,$ and $\delta_{x_0}$ is the Dirac mass at the point $x_0\in\RR.$

     Following the terminology in the linear theory, solutions to ~\eqref{generaleq}-\eqref{eqmdeltainit}
      are called \textit{fundamental solutions}. Another term used for such solutions is \textit{ source-type solutions}.
       The latter is probably better suited, due to the lack of a superposition principle. At any rate, these are solutions evolving from an initial (positive) measure that is located at a single point.

     This paper is concerned with the construction of solutions with measure initial data, generalizing the source-type solutions.

     Equations of the type of ~\eqref{generaleq} are referred to as ``convective-diffusive'' equations. The literature concerning such equations, as well as the related  ``convection-reaction-diffusion'' equations, is quite extensive.

         In the special  case

         \be\label{escvazzua} u_t=u_{xx}-|u|^{q-1}u_x,\quad q>1,\,\,x\in\RR,
         \ee
         the existence and uniqueness of the source-type solution (with initial data ~\eqref{eqmdeltainit}) is proved in ~\cite[Theorem 3]{escobedo}.

       We recall (see ~\cite{escobedo} ) that, in the case of Equation ~\eqref{escvazzua}, for $1<q<2,$  the solution $u(\cdot,t)$ (and, in fact, the solution for every initial function $u_0\in L^1(\RR)$ ) approaches, as $t\to\infty,$ the (self-similar) source-type solution of the hyperbolic equation $u_t=-|u|^{q-1}u_x.$ On the other hand, if $q\geq 2,$ the (nonlinear) convection term becomes negligible and the solution approaches, as $t\to\infty,$  the fundamental solution of the heat equation .

      We note that the long time  decay
     is strongly related to the problem of stability of travelling wave solutions (\cite{ds2} and references therein).


      We are primarily interested in estimates depending only on $\|u_0\|_\Mcal,$ the initial measure norm. For future reference, we make a clear distinction between  estimates that depend on $\eps>0,$ and those that do not.

      For a general flux $f(u)\in C^1(\RR),$ we obtain in Sections ~\ref{secgeneral},~\ref{secnashp} and the first part of Section ~\ref{secLsup}  estimates that depend on $\eps>0.$  On the other hand in Subsection ~\ref{subsecpcondition} we introduce the $p-$condition, a sort of ``weak convexity'' assumption, that has been used in the study of Hamilton-Jacobi equations. This condition allows us to give estimates that are independent of $\eps>0,$ see Corollary ~\ref{corestindepeps}.

      The $p-$condition is used in Theorem ~\ref{theoremuupkestimates} , where we state an existence and uniqueness theorem for solutions of ~\eqref{generaleq}, with measure initial data.

      In Section ~\ref{secutoq} we treat the special case $f(u)=u^q,$ with $q>1.$ This flux satisfies the $p-$condition (in fact $p=q$) , so that the general results can be applied, as well as some additional results depending on this special flux.

 \underline{\textbf{Notation}.}

      For a function $h(x,t)$ we denote $h_t=\frac{\partial}{\partial t}h(x,t)$ and $h_{x}=\frac{\partial}{\partial x}h(x,t).$

      Alternatively we use also $h_t=\partial_t h$ and $h_{x}=\partial_x h.$

      Second-order derivatives are denoted by  $h_{xx}$ or $\partial_x^2h.$

      We use $\RP=(0,\infty).$

  We denote by $\|\cdot\|_q$ the norm in $L^q(\RR),\quad 1\leq
   q\leq \infty.$

   The norm in the measure space $\Mcal_+(\RR)$ is designated as $\|\cdot\|_\Mcal.$

       $W^{k,q}(\RR),\quad 1\leq
   q\leq \infty$ , for a nonnegative integer $k,$ is the space of functions having (distributional) derivatives up to
   order $k$ in $L^q(\RR).$

      $C_0(\RR)$ is the space of continuous, compactly supported functions on $\RR.$

       $C^k(\RR)$ is the space of continuously differentiable  functions on $\RR,$ up to order $k.$

       $C^k_b(\RR)$ is the subspace of $C^k(\RR)$ consisting of all functions whose derivatives up to order $k$ are bounded in $\RR.$

       We write $C_b$ for $C_b^0.$

    \section{GENERAL FACTS for  CONSERVATION LAWS on the REAL LINE}\label{secgeneral}

    In this section we do not assume $u_0\geq 0,$ unless this is  explicitly imposed.

     It is well known that,  under the assumption $u_0\in L^1(\RR)\cap L^\infty(\RR),$ and just $f\in C^1(\RR),$
     Equation ~\eqref{generaleq} has a unique global classical solution $u(x,t),\,\,\,(x,t)\in\RR\times\RP,$
      that converges (in the $L^1$ topology) to $u_0(x)$ as $t\to 0.$
      This solution satisfies the maximum-minimum principle, namely, $-\|u_0\|_\infty\leq u(x,t)\leq \|u_0\|_\infty$
        ~\cite[Section 2.2]{godlewski}

      Another well-known fact is that $\|u(\cdot,t)\|_p$ is nonincreasing, as a function of $t\in[0,\infty),$
      for any $p\in[1,\infty].$

      The ``initial mass'' of the solution $M=\int_{\RR}u_0(x)dx$ is conserved by the evolution,
     \be
     \label{eqfixM}\int_{\RR} u(x,t)dx=M,\quad t>0.
     \ee

      In order to study initial data beyond  $L^1(\RR)\cap L^\infty(\RR),$  we shall need   estimates for the time decay of the norms $\|u(\cdot,t)\|_p,$
       using only  the initial $L^1$ norm $ \|u_0\|_1.$

       A well-known property is the \textit{comparison principle,} as follows.

          If $u_0,\,\,v_0\in
        C_0(\RR)$ are nonnegative initial data, with corresponding solutions $u(x,t),\,v(x,t),$ and if $u_0(x)\leq v_0(x)$ for $x\in \RR,$ then
                    for all $t>0,$
                    \be\label{eqorderpres}
                    u(x,t)\leq v(x,t),\quad x\in \RR.
                    \ee

                     \begin{lem}\label{lemcontraction}
                    Let $u(x,t),\,v(x,t)$ be solutions corresponding to initial functions $u_0,\,v_0\in C_0(\RR),$ respectively. Then,

                    \be\label{eqcontract}
                    \int_{\RR}| u(x,t)-v(x,t)|dx\leq \int_{\RR}| u_0(x)-v_0(x)|dx,\quad t>0,
                    \ee
                    and in particular (taking $v_0=0$)
                    $$ \int_{\RR}| u(x,t)|dx\leq \int_{\RR}| u_0(x)|dx.$$

\end{lem}
          \begin{proof} The properties ~\eqref{eqfixM} and ~\eqref{eqorderpres} allow us to invoke  the Crandall-Tartar lemma
         (~\cite{tartar},~\cite[Section 2.5]{godlewski}), which yields the contraction property ~\eqref{eqcontract}.

          \end{proof}

      We note that the $L^1$ contraction property ~\eqref{eqcontract} satisfied by the solutions to the viscous conservation law can be obtained without resorting to the Crandall-Tartar lemma (and to the order-preserving property ~\eqref{eqorderpres}). Instead, we can use the maximum-minimum principle (for linear equations).

          \begin{lem}\label{lemcontract}

          Let $u(x,t)$ be a solution to  ~\eqref{generaleq}, with  initial data $u_0\in
        C_0(\RR)$ . Then
          \begin{itemize}
          \item The maximum-minimum principle is satisfied by the solution,
                    \be\label{eqmaximum}\inf\limits_{x\in\RR}u_0(x)\leq\inf\limits_{x\in\RR}u(x,t)\leq
                    \sup\limits_{x\in\RR}u(x,t)\leq \sup\limits_{x\in\RR}u_0(x),\quad \forall t>0.\ee

      \item Let $u(x,t),\,\,v(x,t)$ be  solutions to ~\eqref{generaleq}, with respective initial data $u_0,\,v_0\in
        C_0(\RR).$

        Then ,

       \be \label{estcontract}\|u(\cdot,t)-v(\cdot,t)\|_1\leq
           \|u_0-v_0\|_1,\quad t>0.
                  \ee
   \end{itemize}
       \end{lem}
       \begin{proof}
         The maximum-minimum principle is obtained by invoking its validity for the linear convection-diffusion equation. Indeed, consider the linear equation
           $$z_t+f'(u(x,t)) z_x=\eps z_{xx},\quad z(x,0)=u_0(x),$$
           and apply the linear maximum-minimum principle to it.

       To establish the contraction property, let $w(x,t)=u(x,t)-v(x,t).$ It satisfies the equation
       \be
       \label{generaleqdiff}
      w_t +(b(u,v)w)_x=\eps  w_{xx},\quad (x,t)\in\RR\times\overline{\RP},
          \ee
          where $b(u,v)=\frac{f(u)-f(v)}{u-v}.$

          Fix $T>0.$ The dual equation to ~\eqref{generaleqdiff} in the strip $\RR\times[0,T]$ is the
          linear parabolic equation

         \be \label{generaleqdual}
      \phi_t + b(u,v)\phi_x=-\eps  \phi_{xx},\quad (x,t)\in\RR\times[0,T],
          \ee
          subject to the ``terminal'' condition
          $$\phi(x,T)=\phi_T(x)\in C^\infty_0(\RR),$$
          as well as the  boundary condition that
          $$\lim\limits_{R\to\infty}\sup\limits_{|x|>R,\,t\in[0,T]}|\phi(x,t)|=0,.$$
          Clearly $\phi$ satisfies the maximum-minimum principle
          $$\|\phi(\cdot,0)\|_\infty\leq \|\phi_T\|_\infty,$$
          which implies by a standard  duality argument that
          $$\|w(\cdot,T)\|_1\leq \|w(\cdot,0)\|_1.$$
          Since $T>0$ is arbitrary,   ~\eqref{estcontract} is established.
       \end{proof}
 \begin{rem}\label{remlimL11d}Lemma ~\ref{lemcontract} implies that the solution operator
          $$S(t)u_0=u(\cdot,t),\quad t>0,$$
          is a contraction in $L^1,$ hence can be extended to any $u_0\in L^1(\RR).$ However, we have very little information about this extension. In particular, it is not even clear if it is indeed a solution, even in a weak sense, of Equation ~\eqref{generaleq}.
       \end{rem}

       \subsection{FURTHER ESTIMATES }

     In deriving the following estimates, we  assume that  the initial function $u_0$ is
     smooth and compactly supported. This ensures that the  the solution $u(x,t)$   decays at infinity, for any fixed $t>0.$


     The estimates for general initial data will follow by a standard density argument.

     In addition to the contraction property (Lemma ~\ref{lemcontraction}) we have also
     \begin{lem} Assume that $u_0(x)\geq 0.$
     Then for every $p\in[2,\infty),$
                    \be\label{eqcontractp}
                    \int_{\RR} u(x,t)^p\,dx\leq \int_{\RR} u_0(x)^p\,dx.
                    \ee
                  \item The following spacetime estimate holds,
          \be\label{eqspacetime}
         2\eps\int_0^\infty \int_{\RR}|u_x(x,t)|^2\,dxdt\leq\int_{\RR} u_0(x)^2\,dx.
          \ee
     \end{lem}
     \begin{proof}
     By the maximum principle the solution $u(x,t)$ is nonnegative.

To obtain the  estimate ~\eqref{eqcontractp} we multiply Equation ~\eqref{generaleq} by $u^{p-1}$ and integrate over $\RR.$ Noting that
    $$\int_\RR u^{p-1} f(u)_xdx=\int_\RR g(u)_xdx=0,$$
     where $g'(u)=u^{p-1}f'(u),$ we get
     \be\label{eqddtlpnorm}
     d/dt\int_\RR u(x,t)^p\,dx=-\eps p(p-1)\int_\RR u(x,t)^{p-2}| u_x(x,t)|^2\,dx\leq 0.
     \ee
      To obtain the spacetime estimate ~\eqref{eqspacetime} we take $p=2$ in ~\eqref{eqddtlpnorm} and integrate with respect to time.
\end{proof}

\section{ $L^p( \RR)$ ESTIMATES  by the NASH INEQUALITY}\label{secnashp}
 Our treatment of the $L^p$ estimates is  based on the Nash inequality ~\cite{beckner, nash}  restricted  to the one-dimensional case over the whole line. It can be stated as follows.

  Let $\phi$ be an integrable Lipschitz function on $\RR.$
   Then
\begin{equation}
\label{Nash-class-1d} \Big(\int_{\RR}|\phi|^2dx\Big)^{3}\leq
C\int_{\RR}|\phi_x|^2dx\cdot
\Big(\int_{\RR}|\phi|dx\Big)^{4}.
 \end{equation}

            In the context of convection-diffusion equations, the Nash inequality is stated as Lemma 1 in ~\cite{zua}.

            Even though it is not strictly needed, we shall assume in what follows that $u_0\geq 0.$ This will simplify the estimates, as powers of the solution $u$ can be taken without absolute values.

Using the Nash  inequality we obtain the following lemma.

\begin{lem} Let $u(x,t)$ be a solution to ~\eqref{generaleq}, with $u_0\in L^1(\RR).$

    There exists
       a constant $C>0,$  independent of $u_0,\,p,\,\eps,$ such that, for any $2\leq p<\infty$

       \be \label{uMsqrtt1d}\|u(\cdot,t)\|_p\leq
           \Big(\frac{Cp}{\eps}\Big)^{\frac{p-1}{2p}}\|u_0\|_1\,\,t^{-\frac{p-1}{2p}},\quad t>0.
                  \ee
\end{lem}
 \begin{proof} The proof uses the Nash inequality in a way that is essentially identical to the proof of Proposition 1 in ~\cite{zua}. Since it plays an important role in what follows, we bring it here for the convenience of the reader.

   By the well-posedness of ~\eqref{generaleq} in $L^1(\RR)$ we can assume $u_0\in C^1_0(\RR).$

       Multiplying ~\eqref{generaleq} by $pu(x,t)^{p-1},$ where $p\geq 2,$
        and integrating over $\RR$ we get
       \be\label{multequ1d}\aligned
       \frac{d}{dt}\int_\RR u(x,t)^pdx+p\int_\RR u(x,t)^{p-1}
        f(u(x,t))_xdx\\=-\eps p(p-1)\int_\RR u(x,t)^{p-2}| u_x(x,t)|^2dx\\
        =-4\eps\frac{p-1}{p}\int_\RR  |(u(x,t)^{\frac{p}{2}})_x|^2dx.
       \endaligned\ee
         The second integral in the left-hand side of the equality vanishes,
          being the integral of $ F(u(x,t))_x,$ with $F'(\xi)=pf'(\xi)\xi^{p-1}.$

         Invoking the Nash inequality ~\eqref{Nash-class-1d} with $\phi=u(\cdot,t)^{\frac{p}{2}}$ to get
         \be\label{estint1d}\int_\RR |(u(x,t)^{\frac{p}{2}})_x|^2dx\geq C^{-1}\Big(\int_\RR
         u(x,t)^pdx\Big)^{3}
         \cdot\Big(\int_\RR
         |u(x,t)|^{\frac{p}{2}}dx\Big)^{-4}.\ee
         We now use the interpolation inequality
         $$\|g\|_{\frac{p}{2}}\leq\|g\|_p^{\frac{p-2}{p-1}}\|g\|_1^{\frac{1}{p-1}},$$
         with $g=u(\cdot,t).$ Inserting this in ~\eqref{estint1d}
         results in
       \be\label{estint11d}\aligned\int_\RR  (u(x,t)^{\frac{p}{2}})_x^2dx\hspace{200pt}\\
       \geq C^{-1}\Big(\int_\RR
         u(x,t)^pdx\Big)^{1+\frac{2}{p-1}}
         \cdot\Big(\int_\RR
         |u(x,t)|dx\Big)^{-\frac{2p}{p-1}}.\endaligned\ee

         We let $z(t)=\|u(\cdot,t)\|_p^p.$  Recall that
          $\|u(\cdot,t)\|_1\leq \|u_0\|_1. $ Thus, we
         obtain from ~\eqref{multequ1d} and ~\eqref{estint11d},
        \be\label{zztag1d} z'(t)\leq -4C^{-1}\eps\frac{p-1}{p}\|u_0\|_1^{{-\frac{2p}{p-1}}}z(t)^{1+\frac{2}{p-1}}.\ee
           Comparison with the solution of
           $$w'(t)=-4C^{-1}\eps\frac{p-1}{p}\|u_0\|_1^{{-\frac{2p}{p-1}}}w(t)^{1+\frac{2}{p-1}},\,
           w(0)=\infty,$$
            yields

          $$ z(t)\leq \Big(\frac{L_{p}}{\eps}\Big)^{\frac{p-1}{2}}\|u_0\|_1^p\,\,t^{-\frac{p-1}{2}} ,$$
          where
          $$L_{p}=\frac{pC}{8}.$$
          This concludes the proof of  ~\eqref{uMsqrtt1d}.
          \end{proof}

\section{ $L^\infty( \RR)$ ESTIMATES  }\label{secLsup}
        In this section we turn  to $L^\infty( \RR)$ estimates in the one-dimensional case . Here and throughout the rest of the paper we assume that
        \be
        u_0(x)\geq 0.
        \ee

We first recall the sharp estimate of Carlen and Loss ~\cite[Theorem
     1]{loss}, in the case of a scalar conservation law in $\Rn:$

     \begin{equation}\label{generaleqRn}
     u_t +\nabla\cdot f(u)=\eps \Delta u,\quad (x,t)\in\Rn\times\overline{\RP},\quad
  \eps>0.
   \end{equation}

subject to
   \be
     u_0(x)\in L^1(\Rn)\cap L^\infty(\Rn).
   \ee
     In their estimate there is no need to impose a convexity assumption on the flux function $f(u),$ but they impose a regularity   assumption that can be roughly described as:
      \be\label{eqregfcarlen}  \frac{f(\xi)}{\xi}\in C^1(\Rn).
      \ee
     \begin{lem}[Carlen-Loss]\label{lemcarlenloss}
     Assume $  u_0(x)\in L^1( \Rn)\cap  L^\infty( \Rn),$ and that the flux function satisfies ~\eqref{eqregfcarlen}.

     Then
     \be\label{carlenest}
        \|u(\cdot,t)\|_p\leq K(p)\cdot(4\pi\eps t)^{-\frac{n}{2}(1-\frac1p)} \|u_0\|_1,\quad 1\leq p\leq \infty,\,\,t>0,
     \ee
     where $K(p)=\Big(\frac{4\pi}{p}\Big)^{\frac{n}{2p}},\,1<p<\infty,$ and
     $K(1)=K(\infty)=1.$
     \end{lem}
        Taking $n=1$ in the Carlen-Loss estimate   we get
        \be\label{carlenestn1}
        \|u(\cdot,t)\|_\infty\leq (4\pi\eps t)^{-\frac{1}{2}} \|u_0\|_1,\quad t>0.
     \ee

         We note that for the prototypical example $f(u)=u^q$ the assumption ~\eqref{eqregfcarlen} requires $q\geq 2.$

         On the other hand a less optimal (in terms of the coefficient) estimate is obtained in ~\cite[Lemma 3.1]{feir} under the sole assumption $f(u)\in C^1(\RR)$
           (where in fact a much wider class of degenerate convection-diffusion equations is considered):
           \begin{lem}~\cite[Lemma 3.1]{feir}\label{lemlaurencot} Suppose that $f\in C^1(\RR),$ then for some $C>0,$ depending on $\eps,$
            \be\label{eqlaurenc}
              \|u(\cdot,t)\|_\infty\leq Ct^{-\frac{1}{2}} \|u_0\|_1,\,\,\,\,t>0.
            \ee
           \end{lem}

            Thus, actually we can take any $q>1$ in the case $f(u)=u^q$ .

           We observe that this estimate depends on $\eps.$  Furthermore, it is not clear what are the optimal decay estimates, depending possibly on the special features of $f(u).$  This is manifested in
              Theorem ~\ref{basicthm} and Theorem ~\ref{corind} below. Observe that in these theorems the estimates are independent of $\eps,$ in contrast to the  estimates ~\eqref{carlenestn1} and ~\eqref{eqlaurenc}. The decay estimates  in the above mentioned theorems is due to the effect of the nonlinear convective term, an effect
              that is completely absent in the application of the Nash or logarithmic Sobolev inequalities.

              \subsection{CONVEXITY and the $p-$CONDITION }\label{subsecpcondition}

              To establish an $L^\infty$  estimate (that is independent of $\eps>0$), we use the equivalence of the one-dimensional conservation law and the Hamilton-Jacobi equation. The main shortcoming in this approach is that it is based on a \textbf{convexity hypothesis} imposed on $f(u).$

              Slightly more generally, we begin by introducing a certain class of continuously differentiable functions on $[0,\infty)$ as follows (taken from ~\cite{bbl}). Let $\Psi(r)\in C^1[0,\infty)$ be a nonnegative function, having the following property.

           \begin{equation}
\label{Hass}
\begin{cases}

& \bullet\ \mbox{There exists a family of nonnegative smooth functions}\\

& \{\Phim\}_{\eta>0} \quad\mbox{defined in}\quad [0,\infty)\quad
\mbox{such that}\\
(i)&\quad \Phim(0)=0 \quad \mbox{for all} \quad \eta>0.\\
(ii)&\quad \Phim(r^2) \xrightarrow[\eta\rightarrow
{0}^{+}]{}\Psi(r),\quad \mbox{uniformly  in compact intervals of}\quad
[0,\infty).
 \end{cases}
 \end{equation}

\begin{defn}
 \label{pcondi}
 Consider the family of functions $\{\Thetam\}_{\eta>0}$ defined by
$$
\Thetam(r)=2r\Phim'(r)-\Phim(r),\qquad (r,\eta)\in [0,\infty)\times (0,
\infty).
$$
Let $p\in (1,\infty).$ We say that $\Psi$ satisfies the
\textbf{$p$-condition} if there exist $\gamma>0,\quad a>0, \quad b>0,
$ such that, for $r>0$ and sufficiently small $\eta>0,$
\be\label{eqHJThetaeta}
\Thetam(r)\geq
ar^{\frac{p}{2}}-b\eta^\gamma,\quad\mbox{if}\quad p\in(1,\infty).\\
\ee
\end{defn}

Now in addition to our basic assumption ~\eqref{eqassumefr} on $f,$ we impose the following assumption.
\begin{equation}
\label{Hass3}
\qquad \mbox{ For some $p\in (1,\infty)$,}\quad f \quad
\mbox{satisfies the $p$-condition}.
\end{equation}

As was shown in ~\cite{bbl}, the prototypical example

\begin{equation}
\label{special}
f(r)=r^p,\quad p\in (1,\infty),\quad
\end{equation}
satisfies the above assumptions with
$\Phim(r)=(r+\eta^2)^{\frac{p}{2}}-\eta^p \quad.$ In fact, the same
argument shows that one can take
\begin{equation}
\label{specialpl}
f(r)=\sum\limits_{k=1}^m \mu_k r^{p_k},\quad \mu_k>0,
\end{equation}
where  $\set{p_1,...,p_m}\in
(1,\infty)^m,$ and $p=\max\set{p_1,...,p_m}.$
\begin{rem} In the paper ~\cite{bbl} the case $0<p<1$ is also considered. However we note that in this case $f\notin C^1(\RR).$
\end{rem}

\begin{thm}\cite{bbl} \label{boundedviscHJ}
    Consider the equation
 \begin{equation}\label{generaleqHJ}
      v_t + f(v_x)=\eps  v_{xx},\quad (x,t)\in\RR\times\overline{\RP},\quad
   \eps>0,
   \end{equation}
   subject to the initial condition
     \begin{equation}
     \label{initialgenHJ}
      v(x,0)=\varphi(x).
     \end{equation}
Let $p \in (1,\infty), $ and let $ \varphi \in
C_b(\mathbb{R})$. Assume that $f$ satisfies the $p-$condition
 \eqref{Hass3}.
Then the solution $v(x,t)$ satisfies
 \begin{equation}\label{vgradx}
|\partial_x v(x,t)| \leq \|\varphi\|_{\infty}^{\frac{1}{p}} \,
(at)^{-\frac{1}{p}}  \, .
\end{equation}
\end{thm}
   The application of this theorem to Equation ~\eqref{generaleq}
          is straightforward.
     \begin{cor}\label{corestindepeps} Let $u(x,t)$ be the solution to  Equation ~\eqref{generaleq} , subject to the initial condition
     \begin{equation}
     \label{eqinitialviscHJ}
     0\leq u(x,0)=u_0(x)\in L^1( \mathbb{R}).
     \end{equation}
        Assume that $f$ satisfies the $p-$condition
 \eqref{Hass3}.

      Then for all $t>0,$
      \be\label{eqestimateuconserv}
      \|u(\cdot,t)\|_\infty\leq \|u_0\|_1^{\frac{1}{p}} \,
(at)^{-\frac{1}{p}} .
      \ee
     \end{cor}
     \begin{proof} Taking $v(x,t)=\int_{-\infty}^x u(y,t)dy,$ we observe that $v$ satisfies ~\eqref{generaleqHJ}, subject to the initial condition $\varphi(x)=\int_{-\infty}^x u_0(y)dy\in C_b(\RR).$ Thus, the estimate ~\eqref{eqestimateuconserv} follows directly from ~\eqref{vgradx}.
     \end{proof}

    \section{ SOLUTIONS with MEASURE INITIAL DATA}\label{secmeasureinitial}
    We consider again the scalar conservation law ~\eqref{generaleq}.  We now assume that,
    \be\label{initialmeasure}
      u_0(x)\in \Mcal_+.
    \ee

      The existence and uniqueness of  a source-type solution , with $f(u)=u^q,\,q>1,$ and $u_0=M\delta_{x_0},\,M>0,$ are established in ~\cite[Theorem 3]{escobedo}.

          \begin{defn} A continuous function $u(x,t),\,(x,t)\in\RR\times\RP,$ is a \textbf{classical solution} to the general conservation law ~\eqref{generaleq} if the partial derivatives $u_t,\,u_x$ and $u_{xx}$ are continuous and the equation is satisfied pointwise.
          \end{defn}

      \begin{thm}\label{theoremuupkestimates}
      Consider the general conservation law ~\eqref{generaleq} with initial data ~\eqref{initialmeasure},  where $u_0\geq 0$ is a  compactly supported Borel measure on $\RR.$

      Assume that
      \begin{itemize}
      \item \textbf{ASSUMPTION 1:} $f$ satisfies the $p-$condition (Definition ~\ref{pcondi})
 , for some $p>1.$
   \item \textbf{ASSUMPTION 2:}
          \,\,\,There exists a  constant $C>0,$ such that (for the same $p$),
          \be\label{eqassumpgrowthpfder}
        | f'(r)|\leq C(1+r^{p-1}),\quad r>0.
 \ee
     Observe that it implies the growth condition
 \be\label{eqassumpgrowthpf}
         f(r)\leq C(r+r^{p}),\quad r>0.
 \ee
 \end{itemize}
      Then there exists a nonnegative classical solution $u(x,t),\,(x,t)\in\RR\times\RP,$ so that
      \be\label{eqconvtodelta} u(\cdot,t)\xrightarrow[t\to 0+]{}u_0,
      \ee
      the convergence being in the sense of measures.

      In addition, this solution has the following properties.
      \begin{itemize}
      \item \be\label{eqmeqintuxt}
      \int_{\RR}u(x,t)dx\equiv\|u_0\|_\Mcal,\quad t>0.
      \ee
      \item There exists a constant $C>0,$ such that
      \be\label{eqestlinfutpcond}
      \|u(\cdot,t)\|_\infty\leq C\|u_0\|_\Mcal^{\frac1p}t^{-\frac1p},\quad t>0.
      \ee
      \end{itemize}
      \textbf{Uniqueness:} Let $v(x,t)$ be a classical solution satisfying, in the sense of measures,
        \be\label{eqconvtodeltav} v(\cdot,t)\xrightarrow[t\to 0+]{}u_0.
      \ee
         Assume further that, for some constant $C_1>0$ it satisfies the estimate
         \be\label{eqestlinfvtpcond}
      \|v(\cdot,t)\|_\infty\leq C_1\|u_0\|_\Mcal^{\frac1p}t^{-\frac1p}, \quad t>0.
      \ee
       Define the functions
       \be\label{eqdefUV}
       U(x,t)=\int\limits_{-\infty}^x u(y,t)dy,\quad V(x,t)=\int\limits_{-\infty}^x v(y,t)dy.
       \ee
          Assume that the difference  $Z(x,t)=U(x,t)-V(x,t)$ is continuous in the strip $\RR\times[0,\eta]$ for some $\eta>0.$

        Then
        \be u(x,t)\equiv v(x,t).
        \ee
      \end{thm}

          \begin{proof} The solution will be constructed as a limit of regular solutions, obtained by regularizing the singular initial data.

         \underline{ In the first part of the proof, including Claim ~\ref{ClaimHoldcont},}

         \underline{ we do not use the $p-$condition,} just the fact that $f\in C^1(\RR).$  This means that we shall need to deal carefully with estimating spatial and temporal
 derivatives of the approximating sequence.

      Let $\set{u_0^{(k)}}_{k=1}^\infty\subseteq C^\infty_0(\RR)$ be a sequence of nonnegative test functions such that

      \be
      u_0(x)=\lim_{k\to\infty}u_0^{(k)}(x),
      \ee
       where the limit is taken in the sense of measures. If the approximating sequence is obtained by convolving
        $u_0$ with a compactly supported (nonnegative) mollifier, then we can further impose the condition
        \be \sup\limits_{1\leq k<\infty}\|u_0^{(k)}\|_1\leq M:= \|u_0\|_{\Mcal}.
        \ee


   In addition, due to the compact support of $u_0,$ we can assume
       $$supp\,\uzupk\subseteq K,$$
       where $K\Subset\RR $ is a compact interval.

       Let $\uupk(x,t)$ be the solution to ~\eqref{generaleq} subject to the initial condition $u_0^{(k)},$ namely
       \begin{equation}\label{generaleqk}
      \uupk_t + f(\uupk)_x=\eps \uupk_{xx},\quad (x,t)\in\RR\times\overline{\RP},\quad
   \eps>0,
   \end{equation}
        $$\uupk(x,0)=u_0^{(k)}.$$

       In view of Lemma ~\ref{lemcontract} we first have
       \be\label{equupklessM} \|\uupk(\cdot,t)\|_1\leq M,\quad k=1,2,\ldots.\ee

       Let us fix $\tau>0.$ It follows from the $L^\infty$ estimate in Lemma ~\ref{lemlaurencot} that
       \be\Upsilon(\tau)=\sup\limits_{k=1,2,\ldots}\|\uupk(\cdot,\tau)\|_\infty<\infty.\ee
           The maximum principle yields
           \be\label{equnifbdduk}\sup\limits_{k=1,2,\ldots}\|\uupk(\cdot,t)\|_\infty\leq\Upsilon(\tau),\quad t>\tau.\ee
            \begin{rem} Note that some of the estimates in Section ~\ref{secLsup} depend on $\eps>0,$ while some others are $\eps-$independent . For simplicity in what follows, we shall not explicitly mention this dependence.
       \end{rem}

       It follows that
       \be
       \sup\limits_{k=1,2,\ldots}\|\uupk(\cdot,\tau)\|_2<\infty.
       \ee
       In view of ~\eqref{eqspacetime}
        \be\label{equnifl2uxk}
         \eps\sup\limits_{k=1,2,\ldots}\int_\tau^\infty \int_{\RR}|\uupk_x(x,t)|^2\,dxdt<\infty,
          \ee
          hence, by the uniform boundedness ~\eqref{equnifbdduk},

\be\label{equnifl2fuk}
         \eps\sup\limits_{k=1,2,\ldots}\int_\tau^\infty \int_{\RR}|f(\uupk)_x(x,t)|^2\,dxdt<\infty.
          \ee
          The standard $L^2$ theory ~\cite[Section VII.3]{lieberman} now implies that, in every domain
       $P$ such that $\overline{P}$ is compact in $\RR\times[\tau,\infty),$
       \be\label{equnifbdl2derivs}\sup\limits_{k=1,2,\ldots}\int_P\Big[|\uupk_{xx}(x,t)|^2+|\uupk_{t}(x,t)|^2\Big]\,dxdt<\infty.\ee

         The uniform estimates ~\eqref{equnifbdduk} and ~\eqref{equnifbdl2derivs}  by themselves do not imply the (local) convergence of the sequence of solutions  $\set{\uupk}_{k=1}^\infty$ and their derivatives. In general, the $L^\infty$ estimates  should yield, by ``standard parabolic estimates'',  the fact that this set of solutions, along with their first and second-order $x-$derivatives, as well as the first-order $t-$derivatives, are uniformly H\"{o}lder continuous in every domain
       $P$ such that $\overline{P}$ is compact in $\RR\times\RP$ (see e.g. ~\cite[Lemma 4.17]{lieberman}). However, for such estimates to hold one must rely on the fact that the nonlinear term  $f(\uupk)_x$ in ~\eqref{generaleqk} is itself H\"{o}lder continuous in $(x,t),$ and that such estimates are uniform with respect to $k.$  Furthermore, we are not assuming the existence of a second derivative $f''(u).$ Thus, a direct argument seems to be desirable.
       \begin{rem} Note that an application of the theory of general (second-order) parabolic equation to the special case of the viscous conservation law ~\eqref{generaleq} is not straightforward. Specifically, we need first to establish uniform  (with respect to $k$) H\"{o}lder continuity in $(x,t)\in\RR\times[\beta,\infty),$ for every $\beta>0.$ For example, if we take the general nonlinear equation (restricted to one space dimension) in ~\cite[Chapter V]{LSU}, it reads
       $$u_t-\partial_xa(x,t,u,u_x)+b(x,t,u,u_x)=0,$$
       then it covers ~\eqref{generaleq} , with the possibility of $a(x,t,u,u_x)=\eps u_x-f(u)$ and $b\equiv 0$ or $a(x,t,u,u_x)=\eps u_x$ and $b(x,t,u,u_x)=f'(u)u_x.$ However, taking either choice, the constraints imposed in ~\cite[Chapter V, Section 1]{LSU} are not (apriori) satisfied, since they need to hold uniformly for the sequence of derivatives $\set{\uupk_x}_{k=1}^\infty.$ The following claim establishes such uniform estimates.
       \end{rem}

       We formulate the pointwise estimates in the following claim.
                \begin{claim}\label{ClaimHoldcont} The sequences $$\set{\partial_t\uupk(x,t)}_{k=1}^\infty,\,\,\set{\partial_x\uupk(x,t)}_{k=1}^\infty ,\,\,\set{\partial_{xx}\uupk(x,t)}_{k=1}^\infty,$$
                are uniformly bounded  in $(x,t)\in\RR\times [4\tau,\infty).$

                In addition , the sequence $\set{\partial_x\uupk(x,t)}_{k=1}^\infty$
                                is uniformly H\"{o}lder continuous in $(x,t)\in\RR\times [4\tau,\infty),$ with respect to the two variables $x,t.$

                \end{claim}

                 \begin{proof}[Proof of Claim ~\ref{ClaimHoldcont}]

       Let \be\label{eqGeps1D} G_\eps(x,t)=(4\pi \eps t)^{-\frac{1}{2}}\exp(-\frac{x^2}{4\eps t})\ee be the heat
                kernel in $\RR,$ so that, for $t>\tau,$
                \be\label{equupkGeps}\aligned\uupk(x,t)=\int\limits_\RR G_\eps(x-y,t-\tau)\uupk(y,\tau)dy\\-\int\limits_{\tau}^t\int\limits_\RR \Geps(x-y,t-s)
                f(\uupk(y,s))_ydyds\\=\int\limits_\RR G_\eps(x-y,t-\tau)\uupk(y,\tau)dy\\-\int\limits_{\tau}^t\int\limits_\RR \partial_x\Geps(x-y,t-s)
                f(\uupk(y,s))dyds
                .\endaligned\ee
                Differentiating with respect to $x,$
                \be\aligned\label{equxint}\partial_x\uupk(x,t)=\int\limits_\RR \partial_xG_\eps(x-y,t-\tau)\uupk(y,\tau)dy\\-\int\limits_{\tau}^t\int\limits_\RR \partial_x\Geps(x-y,t-s)
                f(\uupk(y,s))_ydyds.\endaligned\ee
                Let, for $t>\tau,$
                 $$A_k(t)= \|\partial_x\uupk(\cdot,t)\|_\infty,\,\,k=1,2,\ldots.$$
                Using the equalities (where $C>0$ is a universal generic constant),
                 \be\label{Gepsx} \|\partial_x\Geps(\cdot,t)\|_\infty=C(\eps t)^{-1},\,\,\|\partial_x\Geps(\cdot,t)\|_1=C(\eps t)^{-\frac12},\quad t>0,\ee
                 we get from ~\eqref{equxint},
                                   for $k=1,2,\ldots$
                 \be\label{eqestAkt}A_k(t)\leq A_0(t)+C_1C\eps^{-\frac12}\int\limits_\tau^t (t-s)^{-\frac12}A_k(s)ds,\,\,\,t>\tau,\ee
                 where, in view of ~\eqref{equnifbdduk},
                 \be A_0(t)= C\Upsilon(\tau)[\eps (t-\tau)]^{-\frac12},\quad t>\tau,\ee
                 and
                 $$ C_1=\sup\limits_{|v|\leq \Upsilon(\tau)}|f'(v)|.$$
                 Shifting the variable $t=\widetilde{t}+\tau$ and defining $\widetilde{A}_k(\widetilde{t})=A_k(t)$ the last estimate can be written as
                 \be
                 \widetilde{A}_k(\widetilde{t})\leq \widetilde{A}_0(\widetilde{t})+C_1C\eps^{-\frac12}\int\limits_0^{\widetilde{t}} (\widetilde{t}-\widetilde{s})^{-\frac12}\widetilde{A}_k(\widetilde{s})d\widetilde{s},
                 \,\,\,\widetilde{t}>0.
                 \ee
                 Defining $\Theta_k(\widetilde{t})=\sup\limits_{0<\widetilde{s}\leq \widetilde{t}}\Big[(\eps\widetilde{s})^\frac12 \widetilde{A}_k(\widetilde{s})\Big],\,\,\,k=0,1,2,\ldots,$

                 we have
                 $$\Theta_k(\widetilde{t})\leq \Theta_0(\widetilde{t})+C_1C{\widetilde{t}}^\frac12\int\limits_0^{\widetilde{t}} (\widetilde{t}-\widetilde{s})^{-\frac12}{\widetilde{s}}^{-\frac12} \Theta_k(\widetilde{s})
                 d\widetilde{s}.$$

                  We now take $\widetilde{T}>0$  so that
                   $$ 2C_1C{\widetilde{T}}^\frac12\int\limits_0^1 (1-u)^{-\frac12}u^{-\frac12}
                                     du<\frac12.$$
                 The last estimate now yields, for $0<\widetilde{t}\leq\widetilde{T},$
                 $$\Theta_k(\widetilde{t})\leq 2\Theta_0(\widetilde{t})\leq C\Upsilon(\tau) , \quad k=1,2,\ldots,$$

                 We conclude that, with $T=\tau+\widetilde{T},$
                 \be\sup\limits_{k=1,2,\ldots}\|\partial_x\uupk(\cdot,t)\|_\infty\leq C\Upsilon(\tau) [\eps (t-\tau)]^{-\frac12},\quad \tau<t<T.\ee

                 In particular, the sequence $\set{\partial_x\uupk(x,t)}_{k=1}^\infty$ is uniformly bounded in $\RR\times [\frac{T+\tau}{2},T].$ Note that $\widetilde{T}$ depends only on $\Upsilon(\tau),$ which is a nonincreasing function of $\tau.$ Thus, we can proceed by uniform $t-$intervals of length $\frac12(T-\tau)$ and obtain
                 a uniform limit
                 \be\label{eqestux}\Upsilon_1(\tau)=\underset{2\tau\leq t< \infty}{\sup\limits_{k=1,2,\ldots}}\|\partial_x\uupk(\cdot,t)\|_\infty<\infty.\ee

                  Now in  addition to ~\eqref{Gepsx} we have
    \be\label{Gepsxx}\|\partial_{x}^2\Geps(\cdot,t)\|_1=C(\eps t)^{-1},\quad t>0,\ee
        that can be used to estimate
            $$\aligned\int\limits_{\RR}|\partial_x\Geps(x+h,t)-\partial_x\Geps(x,t)|dx\hspace{200pt}\\ \leq
            h\int\limits_{\RR}\int\limits_{0}^1|\partial_x^2\Geps(x+\mu h,t)|dxd\mu\leq Ch(\eps t)^{-1},\quad t>0.\endaligned $$
                Hence by interpolation , for any $0\leq \alpha\leq 1,$
                \be\label{eqGepsxxho}
                \int\limits_{\RR}|\partial_x\Geps(x+h,t)-\partial_x\Geps(x,t)|dx\leq C(\eps t)^{-\frac12(1+\alpha)}|h|^\alpha,\quad t>0.
                \ee

                 From ~\eqref{equxint} we obtain, for $t>2\tau,$
                \be\label{equxinta}\aligned \partial_x\uupk(x+h,t)-\partial_x\uupk(x,t)\hspace{100pt}\\ \\=\int\limits_\RR [\partial_xG_\eps(x+h-y,t-2\tau)-\partial_xG_\eps(x-y,t-2\tau)]\uupk(y,2\tau)dy\\-\int\limits_{2\tau}^t\int\limits_\RR [\partial_x\Geps(x+h-y,t-s)-\partial_x\Geps(x-y,t-s)]
                f(\uupk(y,s))_{y}dyds.\endaligned\ee
                The estimate ~\eqref{eqGepsxxho} now yields
                 \be \aligned
               \sup\limits_{k=1,2,\ldots} |\partial_{x}\uupk(x+h,t)-\partial_{x}\uupk(x,t)|\leq \hspace{150pt}\\ C\Upsilon(\tau)\eps^{-\frac12(1+\alpha)}(t-2\tau)^{-\frac12(1+\alpha)}|h|^\alpha\\+C_1\Upsilon_1(\tau)\eps^{-\frac12(1+\alpha)}|h|^\alpha\int\limits_{2\tau}^t(t-s)^{-\frac12(1+\alpha)}ds,\quad t>2\tau.
                \endaligned\ee
                In particular, we obtain the  H\"{o}lder continuity property of the first-order derivative with respect to $x,$
                \be\label{eqHoldux}\aligned
               \sup\limits_{k=1,2,\ldots} |\partial_{x}\uupk(x+h,t)-\partial_{x}\uupk(x,t)|\hspace{100pt}\\\leq [ C\Upsilon(\tau)\tau^{-\frac12(1+\alpha)}+C_1\Upsilon_1(\tau)\tau^{\frac12(1-\alpha)}]\eps^{-\frac12(1+\alpha)}|h|^\alpha,\\ \quad x\in \RR,\,\,t>3\tau,\,\,0<\alpha<1.
                \endaligned\ee

                 Equation ~\eqref{generaleqk} can now be written in the half-plane $t\geq 3\tau$ as

       \begin{equation}\label{generaleqka}
      \uupk_t -\eps \uupk_{xx}= -f(\uupk)_x=a_k(x,t),\quad (x,t)\in\RR\times [3\tau,\infty) \quad
   \eps>0,
   \end{equation}
        where the sequence of continuous functions $\set{a_k(x,t)},\,\,(x,t)\in\RR\times [3\tau,\infty)$ is uniformly bounded and uniformly H\"{o}lder continuous with respect to $x.$ In fact, recalling that $f'(u)$ is assumed to be locally H\"{o}lder continuous (say with exponent $\gamma>0$), we have, for $(x,t)\in\RR\times [3\tau,\infty),$
        \be\aligned
          |a_k(x+h,t)- a_k(x,t)|\leq |f'(\uupk(x+h,t))-f'(\uupk(x,t))|\Upsilon_1(\tau)\\\\+C_1|\partial_x\uupk(x+h,t)-\partial_x\uupk(x,t)|\hspace{70pt}\\\\
          \leq C\Upsilon_1(\tau)^{1+\gamma}|h|^\gamma\hspace{150pt}\\+C_1[ C\Upsilon(\tau)\tau^{-\frac12(1+\alpha)}+C_1\Upsilon_1(\tau)\tau^{\frac12(1-\alpha)}]\eps^{-\frac12(1+\alpha)}|h|^\alpha.
        \endaligned\ee
        that can be written, for some $\delta>0,$
                \be\label{eqakholder}\sup\limits_{k=1,2,\ldots}|a_k(x+h,t)-a_k(x,t)|\leq L|h|^\delta, (x,t)\in \RR\times [3\tau,\infty),\ee
                where $L>0$ depends on $\tau,\eps.$

         The boundedness and uniform H\"{o}lder continuity of the right-hand side terms in ~\eqref{generaleqka} (with respect to $x$) enables us to establish the uniform boundedness of the sequence $\set{\partial_t\uupk(x,t)}_{k=1}^\infty$   (see ~\cite[Ch. IV, Sec.1]{LSU} for the local version), using the Duhamel representation.

         Indeed, writing


 \be\aligned \uupk(x,t)=\int\limits_\RR G_\eps(x-y,t-3\tau)\uupk(y,3\tau)dy\\+\int_{3\tau}^{t}\int\limits_\RR G_\eps(x-y,t-s)
                a_k(y,s)  dyds,\endaligned\ee

        a formal differentiation with respect to $t$ yields,

                 \be\label{equxxint}\aligned \partial_t\uupk(x,t)=\int\limits_\RR\partial_t G_\eps(x-y,t-3\tau)\uupk(y,3\tau)dy\\+a_k(x,t)+\int\limits_{3\tau}^{t}\int\limits_\RR \partial_t\Geps(x-y,t-s)
                [a_k(y,s)-a_k(x,s)]  dyds,\endaligned\ee
               where we have used
                $$\lim\limits_{\eta\to 0}\int\limits_\RR G_\eps(x-y,\eta)a_k(y,t-\eta)dy=a_k(x,t).$$
                Thus  we need only  establish the boundedness  of the spacetime  integral of
              $$I(y,s)= \partial_t\Geps(x-y,t-s)
                [a_k(y,s)-a_k(x,s)] ,$$
                in $\RR\times [3\tau,t].$
                We have
                $$\partial_t G_\eps(x,t)=(4\pi\eps t)^{-\frac12}\Big\{\frac{x^2}{4\eps t^2}-\frac{1}{2t}\Big\}e^{-\frac{x^2}{4\eps t}},$$
                so in view of the estimate ~\eqref{eqakholder} we obtain
                , with $z=\frac{x-y}{2\sqrt{\eps (t-s)}},$
                 $$\int_{3\tau}^{t}\int\limits_\RR|I(y,s)|dyds\leq CL\int_{3\tau}^{t}(t-s)^{-1+\frac{\delta}{2}}ds\int\limits_\RR[z^2+1]e^{-z^2}|z|^\delta dz .$$
                From ~\eqref{equxxint} we now infer that, for $t\in[4\tau,5\tau],$
                $$
               \sup\limits_{k=1,2,\ldots}| \partial_t\uupk(x,t)|<\infty,\quad (x,t)\in\RR\times [4\tau,5\tau].$$
               Since the bound depends only on $\tau$ (and the initial mass $M$), we can proceed by $\tau-$steps to get
                                   the uniform boundedness of the sequence of time-derivatives $\set{\partial_t\uupk(x,t)}_{k=1}^\infty,\,\,(x,t)\in\RR\times [4\tau,\infty):$
               \be\label{equxxintaa}
               \sup\limits_{k=1,2,\ldots}| \partial_t\uupk(x,t)|<\infty,\quad (x,t)\in\RR\times [4\tau,\infty).
               \ee
                        The uniform boundedness of the sequence of second-order spatial derivatives $\set{\uupk_{xx}}_{k=1}^\infty$ now follows from Equation ~\eqref{generaleqka}.

                   The uniform H\"{o}lder continuity of the set of spatial derivatives $\set{\uupk_x}_{k=1}^\infty$ with respect to $t$ follows from the   uniform H\"{o}lder continuity with respect to $x$ by a well-known argument ~\cite[Chapter II, Section 3]{LSU}. We give here the details, since we need to verify that this continuity is uniform in the full half-plane $\RR\times[4\tau,\infty).$

                   Pick $t_2>t_1>4\tau,\,\,x_1<x_2$ and define
                   \be\label{eqdefJk} J^{(k)}(x_1,x_2;t)=\int\limits_{x_1}^{x_2}[\uupk_x(s,t)-\uupk_x(x_1,t_1)]ds.\ee
                   Clearly
                   \be\label{eqJk4difference}\aligned J^{(k)}(x_1,x_2;t_2)-J^{(k)}(x_1,x_2;t_1)\hspace{100pt}\\
                   =\uupk(x_2,t_2)-\uupk(x_1,t_2)-\uupk(x_2,t_1)+\uupk(x_1,t_1).\endaligned\ee

                   The uniform boundedness of $\set{\uupk_t}_{k=1}^\infty$ ~\eqref{equxxintaa} entails, in view of ~\eqref{eqJk4difference}
                   \be\label{eqJ1J2diff}|J^{(k)}(x_1,x_2;t_2)-J^{(k)}(x_1,x_2;t_1)|\leq A|t_2-t_1|,\ee
                   where $A>0$ depends on $\tau,\,\eps,$ but not on $k,\,x_1,\,x_2,\,t_1,\,t_2.$

                   From ~\eqref{eqdefJk} we derive two facts, where we use $A_1>0$ as a generic constant depending on $\tau,\,\eps,$ but not on $k,\,x_1,\,x_2,\,t_1,\,t_2.$
                   \begin{itemize}
                  \item  The uniform  H\"{o}lder continuity of the derivatives $\set{\uupk_x}_{k=1}^\infty$ ~\eqref{eqHoldux} yields
                   $$|J^{(k)}(x_1,x_2;t_1)|\leq A_1|x_2-x_1|^{1+\alpha}.$$

                   \item By the mean value theorem
                   $$J^{(k)}(x_1,x_2;t_2)=[\uupk_x(\sigma,t_2)-\uupk_x(x_1,t_1)](x_2-x_1),\quad \sigma\in[x_1,x_2],$$

                   so again by the  uniform H\"{o}lder continuity of the derivatives $\set{\uupk_x}_{k=1}^\infty$ with respect to $x,$
                   $$\aligned|J^{(k)}(x_1,x_2;t_2)|\geq \Big[|\uupk_x(x_1,t_2)-\uupk_x(x_1,t_1)|-|\uupk_x(\sigma,t_2)-\uupk_x(x_1,t_2)|\Big]|x_2-x_1|\\\geq \Big[ |\uupk_x(x_1,t_2)-\uupk_x(x_1,t_1)|-A_1|x_2-x_1|^{\alpha}\Big]|x_2-x_1|.\endaligned$$
                   \end{itemize}
                   Incorporating these estimates in ~\eqref{eqJ1J2diff} yields
                   $$\aligned|\uupk_x(x_1,t_2)-\uupk_x(x_1,t_1)|\leq \frac{|J^{(k)}(x_1,x_2;t_2)|}{|x_2-x_1|}+A_1|x_2-x_1|^{\alpha}\\ \\
                   \leq \frac{|J^{(k)}(x_1,x_2;t_2)-J^{(k)}(x_1,x_2;t_1)|+|J^{(k)}(x_1,x_2;t_1)|}{|x_2-x_1|}
                   +A_1|x_2-x_1|^{\alpha}\\
                   \leq A\frac{|t_2-t_1|}{|x_2-x_1|}+2A_1|x_2-x_1|^{\alpha}.\endaligned$$
                   Selecting $x_2$ such that $A\frac{|t_2-t_1|}{|x_2-x_1|}=2A_1|x_2-x_1|^{\alpha},$ we obtain the uniform H\"{o}lder continuity of the derivatives $\uupk_{x}$ with respect to $t,$
                   \be |\uupk_x(x,t_2)-\uupk_x(x,t_1)|\leq 2 A |t_2-t_1|^{\frac{\alpha}{1+\alpha}},\quad x\in \RR,\,\,t_2>t_1>4\tau.\ee
                \emph{End of Proof of Claim ~\ref{ClaimHoldcont}}

                \end{proof}

              We now turn back to the proof of Theorem ~\ref{theoremuupkestimates}.

              From  Claim ~\ref{ClaimHoldcont} we infer that the sequence $\set{a_k(x,t)}_{k=1}^\infty$ in Equation ~\eqref{generaleqka} is uniformly  H\"{o}lder continuous, with respect to $(x,t),$ in the half-plane $(x,t)\in \RR\times [4\tau,\infty).$

               We are now able to use the classical Schauder estimates for the heat equation  ~\cite[Chapter 4, Section 2]{LSU} or ~\cite[Chapter 4]{lieberman}, in order to obtain the uniform  H\"{o}lder continuity of the sequences $$\set{\uupk_t(x,t)}_{k=1}^\infty ,\,\,\set{\uupk_{xx}(x,t)}_{k=1}^\infty. $$




       By a diagonal process we can therefore extract a subsequence converging to a function $u(x,t),\,(x,t)\in \RR\times\RP.$ The convergence is uniform, together with all relevant derivatives, in every compact domain $P\Subset \RR\times\RP.$
       It follows that $u(x,t)$ is a classical solution.

       Applying Fatou's lemma to the sequence of nonnegative pointwise converging functions (for every fixed $t>0$) $\set{\uupk(x,t)}_{k=1}^\infty$ and noting ~\eqref{equupklessM} we obtain
       \be\label{eqlessL1uxt} \int_{\RR}u(x,t)dx\leq M=\|u_0\|_{\Mcal}.
       \ee
       It follows that  $u(\cdot,t)\in L^1(\RR)\cap L^\infty(\RR)$ for every $t>0,$ so that all the properties mentioned in Section ~\ref{secgeneral} can be applied. In particular , combining ~\eqref{eqfixM} with ~\eqref{eqlessL1uxt}
       \be\label{eqconstL1uxt}\int_{\RR}u(x,t)dx=const\leq M=\|u_0\|_{\Mcal},\quad t>0.
       \ee

        We now establish the convergence to the initial data in the sense of measures, as in ~\eqref{eqconvtodelta}.

       In the case $f(u)=|u|^q,\,q>1,$ and $u_0(x)=M\delta_0,$ such a proof is given in ~\cite[Section 4]{escobedo}.

       Let $\zeta(x)\in C^\infty_0(\RR). $  For every $k=1,2,\ldots$ we have, by integrating Equation ~\eqref{generaleqk},
         \be\label{eqintformukxt}\aligned\int\limits_\RR\uupk(x,t)\zeta(x)dx-\int\limits_\RR u_0^{(k)}(x)\zeta(x)dx\\=\eps\int\limits_\RR\int\limits_0^t\uupk(x,s)\zeta_{xx}(x)dxds
         +\int\limits_\RR\int\limits_0^tf(\uupk(x,s))\zeta_{x}(x)dxds.\endaligned\ee

         By the contraction estimate ~\eqref{equupklessM} we have
         $$\Big|\int\limits_\RR\int\limits_0^t\uupk(x,s)\zeta_{xx}(x)dxds\Big|\leq M\|\zeta_{xx}\|_\infty t.$$
         Thus
         \be\label{eqconvinitdata}
         \Big|\int\limits_\RR\uupk(x,t)\zeta(x)dx-\int\limits_\RR u_0^{(k)}(x)\zeta(x)dx\Big|\leq M\|\zeta_{xx}\|_\infty t+\int\limits_\RR\int\limits_0^t\Big|f(\uupk(x,s))\zeta_{x}(x)\Big|dxds.
         \ee

        Let us show that, uniformly in $k,$
        \be\label{eqvanishtodoublext}\lim\limits_{t\to 0}\Big|\int\limits_\RR\int\limits_0^tf(\uupk(x,s))\zeta_{x}(x)dxds\Big|=0.
        \ee
%
%
%

      In order to prove it,  the growth assumption ~\eqref{eqassumpgrowthpf} is invoked. \textbf{ Note that it is used here for the first time in the proof.}

      In view of ~\eqref{equupklessM} we  have
       $$ \|\uupk(\cdot,t)\|_1\leq M,\quad k=1,2,\ldots.$$
           Furthermore, the estimate ~\eqref{eqestimateuconserv} yields
           \be\label{eqestuupklinfpcond}
      \|\uupk(\cdot,t)\|_\infty\leq M^{\frac{1}{p}} \,
(at)^{-\frac{1}{p}} ,
      \ee
     so that
       $$\aligned
         \int\limits_\RR\Big|f(\uupk(x,s))\zeta_{x}(x)\Big|dx\leq C\|\zeta_{x}\|_\infty\int\limits_\RR\Big[\uupk(x,s)(1+\uupk(x,s)^{p-1})\Big]dx\\
         \leq CM\|\zeta_{x}\|_\infty\Big[1+M^{\frac{p-1}{p}}(as)^{-\frac{p-1}{p}}\Big],
       \endaligned$$
       from which ~\eqref{eqvanishtodoublext} follows.

       Thus, passing to the limit as $k\to\infty$ in ~\eqref{eqconvinitdata} yields,
       \be \lim\limits_{t\to 0}\Big|\int\limits_\RR u(x,t)\zeta(x)dx-\int\limits_\RR \zeta(x)du_0(x)\Big|=0,\ee
       so that the convergence in measure to the initial data is established.

       A well-known fact about weak convergence of functionals entails
       \be
         \|u_0\|_{\Mcal}\leq \liminf\limits_{t\to 0}\int_{\RR}u(x,t)dx,
              \ee
      and in conjunction with ~\eqref{eqconstL1uxt} we get ~\eqref{eqmeqintuxt}.

       The estimate ~\eqref{eqestlinfutpcond} now follows from ~\eqref{eqestuupklinfpcond}.

       \underline{Finally, we address the uniqueness of the solution.} Let $v(x,t)$ be another classical solution for the same initial data. 

      Let $U(x,t),\,V(x,t)$ be as in ~\eqref{eqdefUV}.

       The functions $U,\,V$ are classical solutions to the viscous Hamilton-Jacobi equation
       \be\label{eqWfW}
        W_t+f(W_x)=\eps W_{xx},\quad x\in\RR,\,\,t>0.
       \ee
       From the convergence in measure ~\eqref{eqconvtodelta} and ~\eqref{eqconvtodeltav} it follows that
       \be\label{eqUVxto0} \lim\limits_{t\to 0}U(x,t)=\lim\limits_{t\to 0}V(x,t)=U_0(x),\quad a.e. x\in\RR,
       \ee
          where
          $$U_0(x)=\int\limits_{-\infty}^xdu_0,$$
          is a monotone nondecreasing function.

          Since $u_0$ is compactly supported, the convergence in measure also implies that for every small $\delta>0$ there exist $r>0,\,\theta>0,$ such that
          \be\label{eqsmallUV}
         \begin{cases} \max (U(x,t),\,V(x,t))\leq \delta,\quad x<-r,\,0<t<\theta,\\\min (U(x,t),\,V(x,t)) \geq 1-\delta,\quad x>r,\,0<t<\theta.
          \end{cases}
          \ee

%
          Let $Z(x,t)=U(x,t)-V(x,t).$  Noting ~\eqref{eqUVxto0} it follows by Helly's theorem ~\cite[Section VIII.4]{natanson} that there exists a decreasing subsequence $\set{t_k\downarrow 0}$ such that
           $$\lim\limits_{k\to\infty}Z(x,t_k)=0,\quad x\in\RR.$$
           By assumption $Z(x,t)$ is continuous in $\RR\times [0,\eta]$ hence $Z(x,0)\equiv 0.$

            The difference $f(U_x)-f(V_x)$ can be written as
          \be\label{eqdiffUxVx}
          f(U_x)-f(V_x)=A(x,t)Z_x(x,t),
          \ee
          where
            $$A(x,t)=\int\limits_0^1f'(\lambda U_x+(1-\lambda)V_x)d\lambda.$$
            Thus $Z(x,t)$ satisfies the linear parabolic equation
            \be\label{eqZpar}
              Z_t+A(x,t)Z_x=\eps Z_{xx},\quad x\in\RR,\,\,t>0.
            \ee
            On account of  assumption ~\eqref{eqestlinfvtpcond} the coefficient $A(x,t)$ is bounded in every strip of the form $\RR\times[\tau,\eta]$ for $0<\tau<\eta.$

            Let $\set{t_k\downarrow 0}$ be a sequence as above and let
            $$M_k=\sup\set{|Z(x,t)|,\quad x\in\RR,\,\,t_k\leq t\leq\eta}.$$

            The maximum principle ~\cite[Section 3.2]{protter} implies that there exists a point $x_k\in\RR$ such that $Z(x_k,t_k)=M_k.$ In view of ~\eqref{eqsmallUV} we may assume that $x_k\in [-r,r],\,\,k=1,2,...$ Hence there is a subsequence (without changing notation) such that $\lim\limits_{k\to\infty}x_k=\bar{x}\in [-r,r]$ and by the assumed continuity of $Z(x,t)$
               $$\lim\limits_{k\to\infty}M_k=Z(\bar{x},0)=0.$$
               Since $\set{M_k}_{k=1}^\infty$ is non-decreasing, we must have $M_k=0,\,k=1,2\ldots$

         \end{proof}

   \section{THE SPECIAL  CASE $f(u)=u^p$}\label{secutoq}
   In this section we consider the special case of a ``power-law'' flux:

   \begin{equation}
   \label{burgers}
   u_t +(u^p)_x=\varepsilon u_{xx},\quad x\in\RR,\quad
   p>1,\quad\eps>0,
   \end{equation}
     subject to the nonnegative measure initial condition
     \begin{equation}
     \label{initial}
   0\leq  u(x,0)=u_0\in  \Mcal_+..
     \end{equation}
     We assume that the measure $u_0$ is compactly supported.

     The flux certainly satisfies the hypotheses imposed in Theorem ~\ref{theoremuupkestimates},  so all the conclusions of the theorem are valid here. In particular, since it clearly satisfies the $p-$condition  , it satisfies the decay estimate ~\eqref{eqestlinfutpcond}.

     We summarize the decay estimates in the following theorem.
%
%
%
       \begin{thm}
       \label{basicthm}
       Let $u(x,t)$ be the solution to ~\eqref{burgers}. Then:
       \begin{enumerate}
       \item With some constant $C>0,$ {\it {independent}} of  \,\,$\varepsilon>0,$
       \begin{equation}
       \label{linfinity}
       \|u(\cdot,t)\|_\infty\leq C\|u_0\|_\Mcal^{\frac1p}t^{-\frac1p},\quad t>0.
       \end{equation}
       \item With some constant $C>0,$ {\it {independent}} of  \,\,$\varepsilon>0,$
        \begin{equation}
       \label{l2}
       \|u(\cdot,t)\|_2\leq C\|u_0\|_\Mcal^{\frac{p+1}{2p}}t^{-\frac{1}{2p}},\quad t>0.
       \end{equation}

       \end{enumerate}
       \end{thm}
       \begin{proof}

      As already noted, the estimate ~\eqref{linfinity} is just the decay estimate ~\eqref{eqestlinfutpcond}.

       The estimate ~\eqref{l2} is obtained by interpolating ~\eqref{linfinity} with
       the contraction property $\|u(\cdot,t)\|_1\leq\|u_0\|_\Mcal.$

      \end{proof}
      \begin{rem} The estimate ~\eqref{linfinity} is proved in ~\cite[Lemma 1.2]{escobedo}, for $1<p<2,$ where the initial data is a point-source ~\eqref{eqmdeltainit}. However, as was seen above, the validity of this estimate also for $p\geq 2$ was useful in studying the behavior of the solution near the initial data.

      For $p\geq 2$ one has the estimate ~\eqref{carlenestn1}, which gives a faster decay as $t\to\infty,$ but depends on $\eps.$
      \end{rem}
      \begin{rem}
      (a) \,It is interesting to compare the $L^2$ estimate ~\eqref{l2} (for the Equation ~\eqref{burgers})
      to the "dispersive
      estimate" ~\cite[Section 1.1]{ds2},
      $$\|u(\cdot,t)\|_2\leq \alpha(\eps)\|u_0\|_1 t^{-\frac14},\quad t>0.$$
       The time decay $t^{-\frac14}$ is identical for the case $p=2$ but is different otherwise.
       The dependence on $\|u_0\|_1$ is different.
       Also note that the constant $C>0$ in ~\eqref{l2} is independent of $\eps>0.$ We
       note, on the other hand, that the "dispersive estimate" is independent of the
       nonlinear term (which is integrated out) and can therefore be applied in other
       situations (see its derivation for the vorticity in two-dimensional
       Navier-Stokes equations in ~\cite[Section 3]{mba}).

       (b)\, The rate of decay (in time) given by ~\eqref{l2} was first derived by
       Schonbek , in the multi-dimensional case ~\cite{sc}. However, the dependence on
       $u_0$ is different, as well as the fact that here the coefficient $C$ is
       independent of $\eps$ (see ~\eqref{l2in} below).

       (c)\, In the case of the viscous Burgers equation ( $p=2$) sharp
       constants for both ~\eqref{linfinity} and ~\eqref{l2} were given in
       ~\cite[Theorem 1]{loss}. The dependence there on $\|u_0\|_1$ is {\emph linear},
       as in the case of the heat equation. However, once again, the coefficients
       depend on $\eps.$

      \end{rem}
          The fact that the constants in ~\eqref{linfinity}-\eqref{l2} are independent of
       $\eps>0$ yields immediately the following result.
       \begin{thm}
       \label{corind}
        Consider the (inviscid) conservation law
   \begin{equation}
   \label{burgersin}
   u_t +(|u|^q)_x=0,\quad x\in\RR,\quad q>1,
   \end{equation}
     subject to the initial condition
     \begin{equation}
     \label{initialin}
     u(x,0)=u_0(x)\in L^1( \mathbb{R}).
     \end{equation}
      Then, with some constant $C>0,$

       \begin{equation}
       \label{linfinityin}
       \|u(\cdot,t)\|_\infty\leq C\|u_0\|_1^{\frac1q}t^{-\frac1q},\quad t>0,
       \end{equation}

        \begin{equation}
       \label{l2in}
       \|u(\cdot,t)\|_2\leq C\|u_0\|_1^{\frac{q+1}{2q}}t^{-\frac{1}{2q}},\quad t>0.
       \end{equation}

       \end{thm}
        \begin{proof}
        Denoting by $\ueps$ the solution to ~\eqref{burgers}, we know from the theory of viscous approximations to hyperbolic conservation laws ~\cite{godlewski} that $\ueps\to u$ pointwise
        for a.e.
        $t>0.$ Therefore ~\eqref{linfinityin} follows from ~\eqref{linfinity}. In
        particular, the set $\set{\ueps(\cdot,t)}_\eps$ is uniformly bounded (for a.e.
        $t$), so that by the dominated convergence theorem
        $$\int\limits_{-N}^N|u(x,t)|^2dx\leq C^2 \|u_0\|_1^{\frac{q+1}{q}}t^{-\frac1q},\quad
        t>0,$$
        and ~\eqref{l2in} follows by letting $N\to\infty.$

        \end{proof}
       \vspace{.2in}
\appendix


\begin{thebibliography}{100}
   \bibitem{aguirre} J. Aguirre and M. Escobedo, {\sl On the blow-up of
solutions of a convective reaction diffusion equation,} Proc. Royal Soc. Edinburgh {\bf 123A} (1993), 433 --460.
   \bibitem{beckner} W. Beckner, {\sl Geometric proof of
Nash's inequality,} IMRN {\bf 2} (1998), 67 --71.
   \bibitem {bbl}  S. Benachour, M. Ben-Artzi  and Ph. Lauren\c{c}ot, {\sl
      Sharp decay estimates and vanishing viscosity for diffusive Hamilton-Jacobi equations},
Adv. Diff. Eqs.  {\bf 14}  (2009), 1-25.
    \bibitem {bkl}  S. Benachour, G. Karch  and Ph. Lauren\c{c}ot, {\sl
      Asymptotic profiles to solutions of convection-diffusion equations},
C.R. Acad. Sci. Paris {\bf 338}  (2004), 369-374.
    \bibitem{mba} M. Ben-Artzi, {\sl Planar Navier-Stokes equations: Vorticity
    approach}, Chapter 5 in {\sl "Handbook of Mathematical Fluid Dynamics, vol. II"} ,
    S. Friedlander and D. Serre, Eds. North-Holland 2003.
    
\bibitem{Loss1} E.A. Carlen and M. Loss, {\sl Sharp constant in
Nash's inequality,} Duke Math. J. {\bf 71} (1993), 213 -- 215.
    \bibitem{loss} E.A. Carlen and M. Loss, {\sl Optimal smoothing and decay estimates
    for viscously damped conservation laws, with applications to the 2-D Navier-Stokes
    equation}, Duke Math. J. {\bf 81} (1995-96), 135-157.
    \bibitem{tartar} M. Crandall and L. Tartar, {\sl Some relations between non-expansive and order preserving mappings,
    }, Proc. AMS {\bf 78} (1980), 385--390.
    \bibitem{zua}M. Escobedo and E. Zuazua, {\sl Large time
    behavior for convection diffusion equations in $\RR^N$}, J.
    Func. Anal. {\bf 100}(1991), 119-161.
    \bibitem{escobedo}M. Escobedo , J. L. Vazquez and E. Zuazua, {\sl Asymptotic
    behavior and source-type solutions for a diffusion-convection  equation}, Arch.
    Rat. Mech. Anal. {\bf 124}(1993), 43-65.
    \bibitem{evans}L.C. Evans, {\sl "Partial Differential Equations", }
 American Mathematical Society , 1998.
    \bibitem{Fabes} E.B. Fabes and D.W. Stroock, {\sl A new proof of
Moser's parabolic Harnack inequality using the old ideas of
Nash}, Arch. Rat. Mech. Anal. {\bf 96}(1986), 327--338.
   \bibitem{feir} E. Feireisl and Ph. Lauren\c{c}ot, {\sl The $L^1$-stability of
   constant states of degenerate convection-diffusion equations}, Asymptotic Anal.
   {\bf 19} (1999), 267-288.
  \bibitem {fr}H. Freist\"{u}hler and D. Serre, {\sl $L^1$ stability of shock waves in
  scalar viscous conservation laws}, Comm. Pure Appl. Math. {\bf 51} (1998), 291-301.
  \bibitem{friedman}A. Friedman, {\sl ``Partial Differential Equations'',}
 Holt, Rinehart and Winston , 1969.
 \bibitem{godlewski} E. Godlewski and P.A. Raviart, {\sl ``Hyperbolic Systems of Conservation Laws
   '',} Ellipses (1990).
  \bibitem{kar} G. Karch and M.E. Schonbek, {\sl On zero mass solutions
   of viscous conservation laws}, Comm. PDE
   {\bf 27} (2002), 2071-2100.
   \bibitem{LSU} O.A. Lady\v{z}henskaja, V. A. Solonnikov and N.N. Ural'ceva, {\sl ``Linear and Quasilinear Equations of Parabolic Type'',} Amer. Math. Soc. Providence (Vol.23 of ``Translations of Mathematical Monographs'') (1968).
       \bibitem{lieberman}G. M. Liebermann, {\sl ``Second Order Parabolic Differential Equations
   '',}World Scientific (1996).
   \bibitem{tpliu}T.-P. Liu and M. Pierre, {\sl Source-solutions and asymptotic behavior
    in conservation laws}, J.
    Diff. Eqs. {\bf 51}(1984), 419--441.
\bibitem{nash} J. Nash, {\sl Continuity of solutions of
parabolic and elliptic equations,} Amer. J. Math. {\bf 80} (1958),
931--954.
 \bibitem{natanson}I.P.Natanson, {\sl ``Theory of Functions of a Real Variable,\,Vol. I'',}
 Ungar Publishing , New York, 1983.
 \bibitem{protter}M. H. Protter and H.F. Weinberger, {\sl ``Maximum Principles in Differential Equations'',}
 Prentice Hall , 1967.
    \bibitem{sc}  M.E. Schonbek, {\sl Decay of solutions
   to parabolic conservation laws}, Comm. PDE
   {\bf 5} (1980), 449-473.
    \bibitem{sc1}  M.E. Schonbek and E. S\"{u}li, {\sl Decay of the total variation
    and Hardy norms of solutions
   to parabolic conservation laws}, Nonlinear Analysis
   {\bf 45} (2001), 515-528.
  \bibitem {ds} D.Serre, {\sl L1-decay and the stability of shock profiles} in {\sl "PDEs. Theory and
numerical solution"},  W. J\"{a}ger, J. Necas, O. John, K. Najzar, J. Stara, eEds.
Pitman RNM 406 , pp 312-321, Chapman and Hall (1999).

   \bibitem{ds1} D. Serre, {\sl Stabilite $L^1$ d'ondes progressives de lois de
   conservation scalaires}, Expos\'{e} VIII in {\sl "\'{E}quations aux
   D\'{e}riv\'{e}es Partielles"}, Ecole Polytechnique 1998-1999.
 \bibitem{ds2} D. Serre, {\sl $L\sp 1$-stability of nonlinear waves in scalar
conservation laws}, in {\sl " Handbook of Differential Equations: Evolutionary
equations. Vol. I"} , 473--553, C. Dafermos , E. Feireisl eds., North-Holland,
Amsterdam, 2004.

\bibitem{souplet} Ph. Souplet and F. B. Weissler, {\sl Poincar\'{e}'s inequality and global solutions of a nonlinear parabolic equation}
    , Annales I. H. P.
   {\bf 16} (1999), 335--371.




  \end{thebibliography}
\end{document}